\documentclass[12pt,reqno]{amsart}
\usepackage{fullpage}
\usepackage{times}
\usepackage{colonequals}
\usepackage{amsmath,amssymb,amsthm,url}
\usepackage[utf8]{inputenc}
\usepackage[english]{babel}
\usepackage{comment}
\usepackage{bbm}
\usepackage{enumerate}
\usepackage{bm}
\usepackage{graphicx}
\usepackage{mathrsfs}
\usepackage[colorlinks=true, pdfstartview=FitH, linkcolor=blue, citecolor=blue, urlcolor=blue]{hyperref}

\newtheorem{thm}{Theorem}[section]
\newtheorem{lem}[thm]{Lemma}

\newtheorem{claim}[thm]{Claim}

\theoremstyle{definition}
\newtheorem{defn}[thm]{Definition}
\newtheorem{question}[thm]{Question}

\newtheorem{rem}[thm]{Remark}

\newcommand{\F}{\mathbb{F}}
\newcommand{\disc}{\operatorname{disc}}

\newenvironment{poc}{\begin{proof}[Proof of the claim]}{\end{proof}}

\title{Mutual position of two smooth quadrics over finite fields}
\author{Shamil Asgarli}
\address{Department of Mathematics \& Computer Science \\ Santa Clara University \\ CA 95050 \\ USA}
\email{sasgarli@scu.edu}

\author{Chi Hoi Yip}
\address{School of Mathematics\\ Georgia Institute of Technology\\ Atlanta, GA 30332\\ United States}
\email{cyip30@gatech.edu}

\begin{document}

\begin{abstract} 
Given two irreducible conics $C$ and $D$ over a finite field $\mathbb{F}_q$ with $q$ odd, we show that there are $q^2/4+O(q^{3/2})$ points $P$ in $\mathbb{P}^2(\mathbb{F}_q)$ such that $P$ is external to $C$ and internal to $D$. This answers a question of Korchm\'{a}ros. We also prove the analogous result for higher-dimensional smooth quadric hypersurfaces in $\mathbb{P}^{n-1}$ with $n$ odd, where the answer is $q^{n-1}/4+O(q^{n-\frac{3}{2}})$.
\end{abstract}

\maketitle

\section{Introduction}

Combinatorics over finite projective planes often involves counting the number of points in the intersection of two geometrically defined subsets. Such subsets include arcs, blocking sets, subplanes, and ovals; when the plane has an odd order, we can also consider a set of external and internal points to an oval. Determining nontrivial estimates of the intersection size is challenging due to the dependency on the relative positions of the two subsets. When working over the projective plane $\mathbb{P}^2(\F_q)$ over a finite field $\F_q$, one can sometimes apply algebraic geometry in positive characteristic to tackle specific counting problems \cite{H79, S97}. 
We refer to related works \cite{AFKL11, DD10, DDk09, PL23}. Many of these concepts have their natural analogs in higher dimensions \cite{AG14, AGY23, CP15}. 

In the present paper, we focus on the problem of counting special points relative to two quadric hypersurfaces (defined by two quadratic forms in $n$ variables) viewed as subsets of $\mathbb{P}^{n-1}$. The following paragraph describes the planar case $n=3$ as a motivation.

Let $q$ be an odd prime power. Let $C\subset \mathbb{P}^2$ be a smooth conic defined over a finite field $\F_q$. Given a point $P\in \mathbb{P}^2(\F_q)$, there are exactly three possibilities:
\begin{enumerate}[(a)]
\item \label{enum:on-curve} $P\in C$ and $C$ admits exactly one tangent line passing through $P$.
\item \label{enum:external} $P\notin C$ and $C$ admits two tangent $\F_q$-lines passing through $P$.
\item \label{enum:internal} $P\notin C$ and $C$ admits no tangent $\F_q$-lines passing through $P$.
\end{enumerate}
We call $P$ an \emph{external point} to $C$ if it satisfies~\eqref{enum:external}, and call $P$ an \emph{internal point} to $C$ if it satisfies~\eqref{enum:internal}. Visualizing a conic as a circle in the Euclidean plane $\mathbb{R}^2$ justifies the names ``internal point" and ``external point". Indeed, no real tangent line passes through a point $P$ inside the circle, and exactly two real tangents to $C$ pass through a point $P$ outside the circle.

Korchm\'{a}ros \cite[Problem 6.2]{K13} asked the following natural question.

\begin{question}\label{quest:korchmaros}
Given two distinct irreducible (equivalently, smooth) plane conics $C$ and $D$ with $q$ odd, how many points in $\mathbb{P}^2(\F_q)$ are external to $C$ but internal to $D$?
\end{question}

In this paper, we provide an answer to Question~\ref{quest:korchmaros}.

\begin{thm}\label{thm:main-conics}
Let $q$ be an odd prime power. Let $C, D$ be two distinct irreducible conics defined over $\F_q$. The number of points $P$ in $\mathbb{P}^2(\F_q)$ external to $C$ but internal to $D$ is $\frac{q^2}{4} + O(q^{3/2})$.
\end{thm}

The prediction in the theorem is intuitive: the ``probability" that a point $P$ is internal or external should be equally likely, leading to a $\frac{1}{4}$ chance that $P$ is external to $C$ but internal to $D$. Since the number of points in $\mathbb{P}^2(\F_q)$ is $q^2+q+1$, the leading term $\frac{q^2}{4}$ in the theorem is reasonable. In the analysis, we should avoid points $P$ that lie on either of the two conics (as the definitions of internal/external do not apply to such points); however, the number of such points is at most $2(q+1)$, which is absorbed by the error term $O(q^{3/2})$. In the next paragraph, we consider the higher-dimensional version of this question.

Consider two smooth quadric hypersurfaces in $\mathbb{P}^{n-1}$, with $n\geq 3$ odd. Given such a quadric $C\subset \mathbb{P}^{n-1}$ defined over $\F_q$, the variety $C$ again satisfies the nice property that a given point $P\in\mathbb{P}^{n-1}(\F_q)$ must fall into one of three categories similar to $\eqref{enum:on-curve}$, $\eqref{enum:external}$, and $\eqref{enum:internal}$. The definitions in higher dimensions require more input from the theory of quadratic forms (or polar spaces), as explained in Section~\ref{sec:internal-external}. Given two smooth quadrics $C$ and $D$ in $\mathbb{P}^{n-1}$, we ask for the number of points in $\mathbb{P}^{n-1}(\F_q)$ external to $C$ and internal to $D$. Using the heuristic mentioned above, we again predict the number of such points to be roughly $\frac{q^{n-1}}{4}$. More generally, we can also count points which are internal / external to $C$ and internal / external to $D$. Specifically, consider the following four sets $S_{1}, S_2, S_3, S_4$:
\begin{align*}
    S_1 &= \{P\in\mathbb{P}^{n-1}(\mathbb{F}_q) \ | \ P \text{ is internal to } C \text{ and internal to } D \}, \\ 
    S_2 &= \{P\in\mathbb{P}^{n-1}(\mathbb{F}_q) \ | \ P \text{ is external to } C \text{ and external to } D \}, \\ 
    S_3 &= \{P\in\mathbb{P}^{n-1}(\mathbb{F}_q) \ | \ P \text{ is internal to } C \text{ and external to } D \}, \\ 
    S_4 &= \{P\in\mathbb{P}^{n-1}(\mathbb{F}_q) \ | \ P \text{ is external to } C \text{ and internal to } D \}. 
\end{align*}
Our main result demonstrates that each of these four sets has size roughly equal to $\frac{q^{n-1}}{4}$.

\begin{thm}\label{thm:main-quadrics}
Let $q$ be an odd prime power and $n \geq 3$ be odd. Let $C$ and $D$ be two distinct smooth quadrics in $\mathbb{P}^{n-1}$ defined over $\F_q$, respectively. 
Then 
$$|S_i|=\frac{q^{n-1}}{4}+O(q^{n-\frac{3}{2}})$$
for each $1\leq i \leq 4$, where the implicit constant in the error term depends only on $n$.
\end{thm}

Note that Theorem~\ref{thm:main-conics} is a particular case of Theorem~\ref{thm:main-quadrics}, so we will only prove Theorem~\ref{thm:main-quadrics}. 

\begin{rem} By keeping track of the explicit bounds, one can compute the constant in the error term $O(q^{n-\frac{3}{2}})$. The explicit constant is at most exponential in $n$ (see Lemma~\ref{lem1}).
\end{rem}

\textbf{Structure of the paper.} In Section~\ref{sec:internal-external}, we introduce the concepts of internal and external points for a smooth quadric. In Section~\ref{sec:character-sums}, we borrow tools from character sum estimates to develop a key technical lemma. We prove Theorem~\ref{thm:main-quadrics} in Section~\ref{sec:proof-main-result}.

\section{Internal and External Points, and Dual quadrics}\label{sec:internal-external}

In this section, we define the concept of an internal and external point to a given quadric hypersurface in $\mathbb{P}^{n-1}$ when $n$ is odd. After computing the discriminant associated with a hyperplane section of a quadric, we give an algebraic criterion for detecting internal/external points. We also explain the relevant notations, recall the classical constructions, and provide further context for our paper. Throughout the section, $q$ is an odd prime power.

\subsection{Internal and external points for higher-dimensional quadrics}\label{subsect:internal-external-higher}

The concept of internal and external points has been studied extensively for plane conics. Generalizations to higher-dimensional quadrics also appear in the literature \cite[Section 1.9]{HT16}; see also Remark~\ref{rem:past-work} for additional references. To motivate the standard definition, we recall the construction of the dual variety from algebraic geometry. Let $X = \{F=0\}$ be a smooth quadric hypersurface of dimension $n-2$ in $\mathbb{P}^{n-1}$. We denote by $(\mathbb{P}^{n-1})^{\ast}$ the parameter space of all hyperplanes in $\mathbb{P}^{n-1}$. Associated with $X$, we have a \emph{dual hypersurface},
$$
X^{\ast} = \{H \in (\mathbb{P}^{n-1})^{\ast} \ | \ H \text{ is tangent to } X \}.
$$
It turns out that $X^{\ast}$ is also a smooth quadric hypersurface. Given an $\F_q$-point $P\in \mathbb{P}^{n-1} \setminus X$, consider the hyperplane $H^{\ast}_{P}\subset (\mathbb{P}^{n-1})^{\ast}$ in the dual space parametrizing those hyperplanes in $\mathbb{P}^{n-1}$ passing through $P$. Observe that the intersection 
$$
H^{\ast}_{P} \cap X^{\ast} 
$$
is a quadric hypersurface $\{G_P = 0\}$ of dimension $n-3$ in $H^{\ast}_{P}\cong \mathbb{P}^{n-2}$. Let $Y_P$ denote the polar space $H^{\ast}_{P}(\F_q)$ equipped with the quadratic form $G_P$. 

\begin{defn}\label{def:internal-external} Let $X$ be a smooth quadric in $\mathbb{P}^{n-1}$ defined over $\F_q$, where $n$ is odd. Let $P \in \mathbb{P}^{n-1}(\F_q)$ be a point not on $X$. We say
\begin{enumerate}[(i)]
\item $P$ is an \emph{internal point} to $X$ if $Y_P$ is elliptic; and
\item $P$ is an \emph{external point} to $X$ if $Y_P$ is hyperbolic.
\end{enumerate}
\end{defn}

The words elliptic and hyperbolic appear in finite geometry, namely in the theory of polar spaces \cite[Section 4.2]{B15}, or quadratic spaces over finite fields \cite[Section 2]{Y24}. More precisely, an $(m-1)$-dimensional projective space over $\F_q$ equipped with a non-degenerate quadratic form in $m$ variables is isometrically isomorphic to a space equipped with one of the following quadratic forms (see \cite[Theorem 3.28]{B15} or \cite[equation (1.2)]{Y24}):
\begin{enumerate}
\item (hyperbolic) $x_1 x_2+...+x_{2k-1} x_{2k}$, where $m=2k$;
\item (elliptic) $x_1 x_2+...+x_{2k-3} x_{2k-2}+(x_{2k-1}^2-\lambda x_{2k}^2)$ for some non-square $\lambda\in\F_q$, where $m=2k$;
\item (parabolic) $x_1 x_2+...+x_{2k-1} x_{2k}+c x_{2k+1}^2$ for some $c\in\F_q^*$, where $m=2k+1$.
\end{enumerate}

The definition of internal and external points in the planar case $n=3$, as given in the introduction, is consistent with Definition~\ref{def:internal-external}. Indeed, the hyperbolic binary form $x_1x_2$ has two zeros $[0:1]$ and $[1:0]$ in $\mathbb{P}^{1}(\F_q)$, corresponding to the two tangent lines to the conic passing through $P$. The elliptic binary form $x_1^2-\lambda x_2^2$ for a non-square $\lambda\in\F_q$ has no zeros in $\mathbb{P}^{1}(\F_q)$, indicating the absence of tangent lines passing through $P$.

As we saw previously, the presence or absence of solutions dictates whether the point is external or internal. To quantify the number of solutions, we use the following concept. The \emph{rank} of a polar space is the dimension of a maximum totally isotropic subspace. For a polar space equipped with a non-degenerate quadratic form with $2k$ variables, the rank is $k$ in the hyperbolic case and $k-1$ in the elliptic case. When the number of variables is $2k + 1$, the only possibility is the parabolic case, where the rank is $k$ \cite[Table 4.1]{B15}. Since hyperbolic spaces have larger rank than elliptic ones, the quadratic form in the hyperbolic case has more zeros over $\F_q$ compared to the elliptic case (by applying induction on the dimension of the space; see, for example, \cite{P51}), justifying Definition~\ref{def:internal-external}. 

When $n$ is even, the polar space $Y_P$ is equipped with a quadratic form with $m=n-1$ variables. Thus, $Y_P$ is parabolic for every $P\notin X$, and it is unclear how to classify the $\F_q$-points of $\mathbb{P}^{n-1}\setminus X$ into the two categories such as internal and external. Indeed, all quadratic forms associated with a parabolic polar space have the same number of $\F_q$-points.

When $n$ is odd, to distinguish the two cases in Definition~\ref{def:internal-external}, it suffices to study the \emph{discriminant} of the quadratic form $G_P$. The discriminant of a non-degenerate quadratic form $\sum_{1\leq i, j\leq m} a_{ij} x_i x_j$ in $m$ variables is defined by the determinant of the matrix $(a_{ij})$ assuming that $a_{ji}=a_{ij}$ for $i\neq j$. For the associated quadratic hypersurface, the discriminant is well-defined only inside the quotient $\F^{\ast}_{q}/(\F^{\ast}_q)^2$, where $m$ is necessarily even; thus, we view the discriminant as an element in $\F^{\ast}_{q}/(\F^{\ast}_q)^2$. Lemma~\ref{lem:discriminant} allows us to detect whether the form is hyperbolic or elliptic based on the value of the discriminant. The next subsection delves further into the specific properties of the discriminant. 

\subsection{Discriminant of the hyperplane section of a quadric}

We will prove an explicit formula for the discriminant of the quadratic form $G_P$ defined in Section~\ref{subsect:internal-external-higher}.

\begin{lem}\label{lem:G_p}
The discriminant of $G_p$ is $F(P)/\disc(F)$.
\end{lem}

Since $X=\{F=0\}$ is smooth, $F$ is a non-degenerate quadratic form. After an invertible linear change of variables, we may assume that $F(\mathbf{x})=\alpha_0 x_0^2+\alpha_1x_1^2+\cdots+\alpha_{n-1}x_{n-1}^2$ \cite[equation (1.1)]{Y24} with $\alpha_i \neq 0$ for each $0 \leq i \leq n-1$. Note that the discriminant of a quadratic form is invariant under such a change of variables. Let $A$ be the diagonal matrix $(\alpha_i)$; then we can write $F(\mathbf{x})=\mathbf{x}A\mathbf{x}^t$ where $\mathbf{x}$ is a row vector.

We describe the points of the dual hypersurface of $X$. 
\begin{align*}
X^*
&=\bigg\{\nabla F (Q)=\bigg[\frac{\partial F}{\partial x_0}(Q): \frac{\partial F}{\partial x_1}(Q): \cdots: \frac{\partial F}{\partial x_{n-1}}(Q)\bigg]: Q \in X\bigg\}\\
&=\{(A[q_0, q_1, \ldots, q_{n-1}]^t)^t: Q=[q_0:q_1:\cdots:q_{n-1}] \in X\}\\ 
&=\{[q_0, q_1, \ldots, q_{n-1}] A: Q=[q_0:q_1:\cdots:q_{n-1}] \in X\}\\ 
&=\{\mathbf{y} \in \mathbb{P}^{n-1}: F(\mathbf{y}A^{-1})=0\}=\{\mathbf{y} \in \mathbb{P}^{n-1}: \mathbf{y}A^{-1}\mathbf{y}^t=0\}.
\end{align*}

Let $P=[p_0:p_1:\cdots:p_{n-1}] \in \mathbb{P}^{n-1}$. Without loss of generality, assume that $p_0=1$.
$$
H_P^*=\bigg\{\mathbf{y}=[y_0:y_1:\cdots:y_{n-1}] \in \mathbb{P}^{n-1}: \sum_{i=0}^{n-1} y_ip_i=0\bigg\}.
$$
Let $\mathbf{y}=[y_0:\mathbf{z}]$ with $\mathbf{z}=[y_1:y_2:\ldots: y_{n-1}]$. If $\mathbf{y} \in H_{P}^*$, then we have $\mathbf{y}^t=B\mathbf{z}^t$, where
$$
B=
\begin{pmatrix} 
-p_1&-p_2&\cdots &-p_{n-1}\\
1&0&\cdots&0\\
0&1&\cdots&0\\
\vdots&\vdots&\ddots&\vdots\\
0&0&\cdots&1
\end{pmatrix}
$$
is an $n\times (n-1)$ matrix. It follows that 
\begin{align*}
H_P^* \cap X^*
&=\{\mathbf{y}=[y_0:\mathbf{z}] \in \mathbb{P}^{n-1}: \mathbf{y}A^{-1}\mathbf{y}^t=0, \mathbf{y}^t=B\mathbf{z}^t\}\\
&\cong \{\mathbf{z} \in \mathbb{P}^{n-2}: \mathbf{z}B^tA^{-1}B\mathbf{z}^t=0\},
\end{align*}
where we identify $H_{P}^{\ast}$ as $\mathbb{P}^{n-2}$ with homogeneous coordinates given by $\mathbf{z}$. Thus, $H_P^{\ast}\cap X^{\ast}$ is a quadratic hypersurface defined by the following quadratic form:
$$
G_p(\mathbf{z})=\mathbf{z}B^tA^{-1}B\mathbf{z}^t.
$$

Now, Lemma~\ref{lem:G_p} immediately follows from the lemma below by setting $a_i=\alpha_{i}^{-1}$ for each $0\leq i\leq n-1$.
\begin{lem}
Let $a_0, a_1, \ldots, a_{n-1} \in \F_q^*$. Let $c_0=1$, and let $c_1,c_2,\ldots, c_{n-1} \in \F_q$.  Let $A_n$ be the diagonal matrix with diagonal entry $(a_i^{-1})_{i=0}^{n-1}$, 
$$
B_n=\begin{pmatrix} \mathbf{c} \\ I_n \end{pmatrix} = 
\begin{pmatrix} 
c_1&c_2&\cdots &c_{n-1}\\
1&0&\cdots&0\\
0&1&\cdots&0\\
\vdots&\vdots&\ddots&\vdots\\
0&0&\cdots&1
\end{pmatrix},
$$
and $E_n=B_n^{t} A_n^{-1}B_n$. Then $$
\det(E_n)=a_0a_1\cdots a_{n-1} \cdot \sum_{i=0}^{n-1} \frac{c_i^2}{a_i}.
$$
\end{lem}
\begin{proof}
Note that
$$
M_n=B_n^t A_n^{-1} = \begin{pmatrix}
a_0 c_1 & a_1 & 0 & \cdots  & 0 \\ 
a_0 c_2 & 0 & a_2 & \cdots  & 0 \\ 
\vdots & \vdots & \vdots &\ddots & \vdots \\  
a_0 c_{n-1} & 0 & 0 & \cdots  & a_{n-1} \\ 
\end{pmatrix}.
$$
We apply the Cauchy-Binet formula to calculate the determinant of the matrix $E_n = M_n B_n$. Since $M_{n}$ is an $(n-1)\times n$ matrix and $B_n$ is an $n\times (n-1)$ matrix, we have $\binom{n}{n-1}=n$ summands in the Cauchy-Binet formula. More precisely,
$$
\det(E_n) = \sum_{i=0}^{n-1} \det(M^{(i)}_{n}) \cdot \det(B^{(i)}_{n}),
$$
where $M^{(i)}_{n}$ and $B^{(i)}_{n}$ are the matrices obtained from $M_n$ and $B_n$ by removing the $(i+1)$-th column and $(i+1)$-th row, respectively. For $i=0$, we have $M^{(0)}_{n}=\operatorname{diag}(a_1, ..., a_{n-1})$ and $B^{(0)}_{n} = I_n$, so
$$
\det(M^{(0)}_{n}) \cdot \det(B^{(0)}_{n}) = a_1a_2\cdots a_{n-1} = (a_0 a_1 ... a_{n-1}) \cdot \frac{c_0^2}{a_0} 
$$
as $c_0=1$. When $i=1$, we have
$$
M^{(1)}_{n} =  
\begin{pmatrix}
a_0c_1 & 0 & \cdots  & 0 \\ 
a_0c_2 & a_2 & \cdots  & 0 \\ 
\vdots & \vdots &\ddots & \vdots \\  
a_0c_{n-1} & 0 & \cdots  & a_{n-1} \\ 
\end{pmatrix} \quad \text{ and } \quad 
B^{(1)}_{n} = 
\begin{pmatrix}
c_1 & c_2 & \cdots  & c_{n-1} \\ 
0 & 1 & \cdots  & 0 \\ 
\vdots & \vdots &\ddots & \vdots \\  
0 & 0 & \cdots  & 1 \\ 
\end{pmatrix}.
$$
Hence, $\det(M^{(1)}_n)\det(B^{(1)}_n) = (a_0 a_2\cdots a_{n-1} c_1)\cdot (c_1) = (a_0 a_1\cdots a_{n-1})\cdot \frac{c_1^2}{a_1}$. Similar computation for $i>1$ shows that $\det(M^{(i)}_n) \det(B^{(i)}_n) = (a_0a_1\cdots a_{n-1})\frac{c^2}{a_i}$. Thus, 
$$
\det(E_n) = \sum_{i=0}^{n-1} \det(M^{(i)}_{n}) \cdot \det(B^{(i)}_{n}) = a_0a_1\cdots a_{n-1} \cdot \sum_{i=0}^{n-1} \frac{c_i^2}{a_i}
$$
as desired. \end{proof}

\subsection{An algebraic criterion for internal/external points}

The following lemma is folklore; see \cite[Proof of Proposition 6.1]{BIK23}. In particular, Lemma~\ref{lem:discriminant} recovers the classical algebraic criterion for the case $n=3$; this special case for irreducible conics appears in both \cite[Theorem 8.3.3]{H79} and \cite[Result 1]{S92}. We include a proof for the sake of completeness. 

\begin{lem}\label{lem:discriminant}
Let $n$ be odd. Let $X=\{F=0\}$ be a smooth quadric over $\F_q$ given by the equation $F(x_1, \ldots, x_{n}) = \sum_{1\leq i, j \leq n} a_{ij} x_i x_j =0$ such that $a_{ij}=a_{ji}$ for $i\neq j$. Let $\Delta=\det(a_{ij})$. Given a point $P=[p_1: \ldots: p_{n}] \in \mathbb{P}^{n-1}(\F_q)$, we have:
\begin{enumerate}[(a)]
    \item $P\in X$ if and only if $F(P) = 0$.
    \item $P$ is external to $X$ if and only if $(-1)^{(n-1)/2}\Delta\cdot F(P)$ is a non-zero square in $\F_q$.
    \item $P$ is internal to $X$ if and only if $(-1)^{(n-1)/2}\Delta\cdot F(P)$ is a non-square in $\F_q$.
\end{enumerate}    
\end{lem}
\begin{proof}
Let $n=2k+1$. By Definition~\ref{def:internal-external}, to check if $Y_P$ is hyperbolic or elliptic, it suffices to compute the discriminant of the associated quadratic form $G_P$. Note that if $Y_P$ is hyperbolic, then $\disc(G_P)=(-1)^k$; if $Y_P$ is elliptic, then $\disc(G_P)=(-1)^{k-1} (-\lambda)=(-1)^k \lambda$, where $\lambda$ is a non-square in $\F_q$. Note that $(-1)^k \neq (-1)^k \lambda$ are different elements in the quotient group $\F^{\ast}_{q}/(\F^{\ast}_q)^2$. 

Since $X$ is smooth, $\Delta=\disc(F) \neq 0$. We have shown that $\disc(G_p)=F(P)/\Delta=\Delta \cdot F(P)$ in $\F^{\ast}_{q}/(\F^{\ast}_q)^2$ from Lemma~\ref{lem:G_p}. Thus, $P$ is external to $X$ if and only if $\disc(G_p)=(-1)^k$, that is, $(-1)^{(n-1)/2}\Delta\cdot F(P)$ is a non-zero square in $\F_q$.
\end{proof}

\begin{rem}\label{rem:past-work}    
Carlitz \cite{C66} and Jung \cite{J72} defined \emph{exterior} and \emph{interior} points of a quadric $X=\{F=0\}$ over $\mathbb{P}^{n-1}$ in a different way. If $P=[p_1: \cdots :p_n] \in \mathbb{P}^{n-1}(\F_q)$, then $P$ is exterior to $X$ if $F(p_1, \ldots, p_n)$ is a non-zero square in $\F_q$, and interior to $X$ if $F(p_1, \ldots, p_n)$ is a non-square in $\F_q$.  However, as pointed out by Primrose \cite{P72} in the MathSciNet review of \cite{J72}, this terminology lacks geometric meaning since it depends on the equation of $F$ (because an interior point for $F$ becomes an exterior point for $cF$ if $c$ is a non-square). On the other hand, Bruno \cite[Section 9]{B68} used the statement of Lemma~\ref{lem:discriminant} as the definition of exterior and interior points but did not offer a geometric interpretation. 
\end{rem}

\section{Proof of the main result}

\subsection{Technical lemmas from character sum estimates}\label{sec:character-sums}

We begin with two technical lemmas involving character sum estimates. The first one is a simplified version of a well-known result due to Katz \cite[Theorem 2.2]{K02} on nonsingular multiplicative character sum estimates.

\begin{lem}[Katz]\label{lem:Katz}
Let $q$ be a prime power and $\chi$ be a nontrivial multiplicative character of $\F_q$. Let $f\in \F_q[x_1,x_2,\ldots, x_n]$ be a homogeneous polynomial with degree $d\geq 1$, such that the equation $f=0$ defines a smooth hypersurface in $\mathbb{P}^{n-1}$. If $\chi^d$ is trivial, then 
$$
\left|\sum_{\mathbf{x}\in \F_q^n} \chi(f(\mathbf{x}))\right| \leq (d-1)^n q^{n/2}.
$$    
\end{lem}

We also need the following general result on estimating multiplicative character sums over possibly singular spaces due to Rojas-Le\'{o}n \cite[Theorem 1.1(a)]{RL05}. We state a simplified version of Rojas-Le\'{o}n's theorem. 

\begin{lem}[Rojas-Le\'{o}n]\label{lem:RL}
 Let $X$ be a projective Cohen-Macaulay scheme defined over $\F_q$ with pure dimension $m \geq 2$, embedded in $\mathbb{P}^N$ as the closed subscheme defined by $r$ homogeneous forms $F_1, \ldots, F_r$ of degrees $a_1, \ldots, a_r$. Let $H$ and $Z$ be homogeneous forms in $\F_q\left[X_0, \ldots, X_N\right]$ of degrees $d$ and $e$, where $\gcd(d, e)=1$, $\gcd(e, p)=1$ with $p$ the characteristic of $\F_q$. We will also denote by $H$ and $Z$ the hypersurfaces they define in $\mathbb{P}^N$. Assume that $X \cap H \cap Z$ has pure codimension 2 in $X$, and denote by $\delta$ the dimension of the singular locus of $X \cap H \cap Z$. Let $\chi$ be a nontrivial multiplicative character of $\F_q$, such that $\chi^e$ is nontrivial. Let $V=X-(H \cup Z)$ and $f\colon V \rightarrow \F_q^*$ be the map defined by $f(\mathbf{x})=H(\mathbf{x})^e / Z(\mathbf{x})^d$. The following estimate holds:
$$
\bigg|\sum_{\mathbf{x} \in V(\F_q)} \chi(f(\mathbf{x}))\bigg| \leq 3(3+\max \left(a_1, \ldots, a_r, e\right)+d)^{N+r+2} \cdot q^{(m+\delta+2) / 2} .
$$   
\end{lem}

In the rest of the paper, let $q$ be an odd prime power and let $\chi$ be the quadratic character of $\F_q$, that is, 
$$
\chi(a)=\begin{cases}
1  &\text{if } a \text{ is a nonzero square in $\F_q$,} \\ 
-1 &\text{if } a \text{ is a non-square in $\F_q$,} \\ 
0 &\text{if } a=0,
\end{cases}
$$
for each $a\in\F_q$.

\begin{lem}\label{lem1}
Let $q \geq 7$ be an odd prime power and $n \geq 3$ odd. Assume that $C:f=0$ and $D:g=0$ are two distinct smooth quadrics in $\mathbb{P}^{n-1}$, defined over $\F_q$. Then
$$
\left|\sum_{P\in\mathbb{P}^{n-1}(\F_q)} \chi(f(P) g(P))\right| \leq 3\cdot 8^{n+1} q^{(2n-3)/2} +2q^{n-2}. 
$$ 
\end{lem}

\begin{proof} We apply Lemma~\ref{lem:RL} by choosing $X$, $H$, and $Z$ as follows:
\begin{enumerate}[(i)]
\item $X=\mathbb{P}^{n-1}$, $m=n-1$, $N=n-1$, and $r=0$;
\item $H(\mathbf{x})=f(\mathbf{x})g(\mathbf{x})$, and $d=4$;
\item $Z$ is a hyperplane $(e=1)$ which is neither tangent to $C=\{f=0\}$, neither tangent to $D=\{g=0\}$, and nor $Z$ contains any irreducible component of $C\cap D$. 
\end{enumerate}   
We justify why such $Z$ exists. The number of tangent hyperplanes defined over $\mathbb{F}_q$ to either $C$ or $D$ is at most $2(q^{n-2}+q^{n-3}+\cdots+1)$. Moreover, $C\cap D$ has at most $4$ distinct irreducible components, each with dimension $n-3$. The number of hyperplanes in $\mathbb{P}^{n-1}$ containing a linear component (of dimension $n-3$) is $q+1$; if the component is not linear, there is at most one hyperplane in $\mathbb{P}^{n-1}$ containing it. Thus, the number of ``bad" hyperplanes is at most $2(q^{n-2}+\cdots + 1) + 4(q+1)$. However, the total number of hyperplanes is $q^{n-1}+q^{n-2}+\cdots+1$, which is strictly bigger than $2(q^{n-2}+\cdots + 1) + 4(q+1)$ when $q\geq 7$ and $n \geq 3$. In fact, when $n\geq 5$, the intersection $C\cap D$ does not contain a hyperplane component because $C$ is parabolic (the maximum projective dimension of a linear space within $C$ is $\frac{n-3}{2}<n-3$); in this case, the number of ``bad" hyperplanes can be reduced from $2(q^{n-2}+\cdots + 1) + 4(q+1)$ to $2(q^{n-2}+\cdots + 1) + 2$. As a result, one can check that the hyperplane $Z$ exists when $q\geq 3$ and $n\geq 5$.

Observe that $H(\mathbf{x})=f(\mathbf{x})g(\mathbf{x})$ defines a quartic hypersurface. For one of the hypotheses, $X\cap H\cap Z = H\cap Z$ has pure codimension $2$ in $X=\mathbb{P}^{n-1}$ since a hyperplane section of a quartic is a quartic of one less dimension (hence pure codimension $2$ in $X$). Since the hyperplane $Z\cong\mathbb{P}^{n-2}$ does not contain any components of $C\cap D$, the dimension of $C\cap D\cap Z$ is $\delta = n-4$. Indeed, after a projective linear change of variables, the hyperplane $Z$ is defined by the equation $\{x_{n-1}=0\}$; then, $C_0 = C\cap Z$ and $D_0 = D\cap Z$ are two quadrics of one less dimension (with the same equations as $C$ and $D$ after substituting $x_{n-1}=0$). Since $Z$ is tangent to neither $C$ nor $D$, we have that $C_0$ and $D_0$ are smooth quadrics. Moreover, as $Z$ does not contain any irreducible component of $C\cap D$, it follows that $C_0\neq D_0$, and so $C_0\cap D_0 = C\cap D\cap Z$ has the expected dimension of $n-4$.

\begin{claim}
The singular locus of $X\cap H\cap Z$ is equal to $C\cap D\cap Z$; in particular, the dimension of the singular locus of $X\cap H\cap Z$ is $\delta = n-4$.
\end{claim}
\begin{poc}
By abuse of notation, we write $Z(x_0, \ldots, x_{n-1})=0$ for the defining equation of the hyperplane $Z$. The singular locus of $X\cap H\cap Z = H\cap Z$ is the set of points $P$ in $H\cap Z$ such that the following $2\times n$ Jacobian matrix has rank less than $2$:
$$
\operatorname{Jac}_{P}(f(\mathbf{x})g(\mathbf{x}), Z(\mathbf{x})) = 
\begin{pmatrix}
 f(P) \frac{\partial g}{\partial x_0}(P) + g(P) \frac{\partial f}{\partial x_0}(P) &  \cdots &  f(P) \frac{\partial g}{\partial x_{n-1}}(P) + g(P) \frac{\partial f}{\partial x_{n-1}}(P) \\ 
 \frac{\partial Z}{\partial x_0}(P) & \cdots & \frac{\partial Z}{\partial x_{n-1}}(P)
\end{pmatrix}.
$$
Since $P\in H\cap Z$, it follows that $f(P)=0$ or $g(P)=0$. We consider the following three cases. 
\begin{enumerate}
    \item If $f(P)=0$ and $g(P)=0$, then the first row vanishes, and the matrix has rank $1$. 
\item If $f(P)=0$ and $g(P)\neq 0$, then the first row is a nonzero scalar multiple of the gradient vector $\nabla f (P)=\left(\frac{\partial f}{\partial x_0}(P), \ldots, \frac{\partial f}{\partial x_{n-1}}(P) \right)$. In this case, the first row is not a multiple of the second row due to the condition (iii), as the hyperplane $Z$ is not tangent to the quadric $C=\{f=0\}$. Hence, the matrix has full rank (i.e., rank $2$) in this case.
\item If $f(P)\neq 0$ and $g(P)=0$, the same conclusion as in (2) holds by symmetry. 
\end{enumerate}
Therefore, the singular locus of $H\cap Z$ is precisely $C\cap D\cap Z$.
\end{poc}

By Lemma~\ref{lem:RL}, we have
\begin{equation}\label{eq:-Z}
\left|\sum_{P\in\mathbb{P}^{n-1}(\F_q)\setminus Z} \chi(f(P) g(P))\right|=\left|\sum_{P\in\mathbb{P}^{n-1}(\F_q)\setminus (H \cup Z)} \chi\bigg(\frac{f(P) g(P)}{Z(P)^4}\bigg)\right| \leq 3\cdot 8^{n+1} q^{(2n-3)/2}.  
\end{equation}
On the other hand, we have the trivial upper bound 
\begin{equation}\label{eq:Z}
\left|\sum_{P\in Z(\F_q)} \chi(f(P) g(P))\right| \leq q^{n-2}+q^{n-3} + \cdots + q + 1 < 2 q^{n-2}.
\end{equation}
Combining inequality~\eqref{eq:-Z} and inequality~\eqref{eq:Z}, we get
$$
\left|\sum_{P\in\mathbb{P}^{n-1}(\F_q)} \chi(f(P) g(P))\right| \leq 3\cdot 8^{n+1} q^{(2n-3)/2} +2q^{n-2}, 
$$
as required.
\end{proof}

\subsection{Proof of Theorem~\ref{thm:main-quadrics}}\label{sec:proof-main-result}
The main result Theorem~\ref{thm:main-quadrics} holds for $q<7$ because we can always increase the implicit constant in the error term. Hence, we assume $q\geq 7$ for the rest of the paper. Let $C=\{f=0\}$ and $D=\{g=0\}$ denote two smooth quadrics in $\mathbb{P}^{n-1}$. Set $n=2k+1$ for $k\in\mathbb{N}$. We only estimate the size of $S_{f,g}\colonequals S_4$; the proof for the other three sets is similar.

Recall that the set $S_{f, g}$ consists of all points $P$ in $\mathbb{P}^{n-1}(\F_q)$ such that $P$ is external to $C=\{f=0\}$, and $P$ is internal to $D=\{g=0\}$. Let $A$ and $B$ be the discriminants of the quadratic forms $f$ and $g$, respectively. By Lemma~\ref{lem:discriminant}, a point $P$ belongs to $S_{f, g}$ if and only if
$$
\chi(f(P)g(P))\neq 0, \quad \quad \chi((-1)^k A f(P))=1, \ \quad \quad \text{ and } \quad  \chi( (-1)^k B g(P))=-1,
$$
Thus, we have
$$
\#S_{f, g} = \sum_{\mathbf{x}\in\mathbb{P}^{n-1}(\F_q)} \chi(f^2(\mathbf{x}) g^2(\mathbf{x})) \cdot \left(\frac{1+\chi((-1)^kA\cdot f(\mathbf{x}))}{2} \right) \cdot \left(\frac{1-\chi((-1)^kB\cdot g(\mathbf{x}))}{2} \right).
$$
After simplifying and using properties of $\chi$, we obtain the following expression for $\# S_{f, g}$:
\begin{align*}
& \frac{1}{4} \sum_{\mathbf{x}\in\mathbb{P}^{n-1}(\F_q)} \chi(f^2(\mathbf{x}) g^2(\mathbf{x})) + \chi((-1)^kA f^3(\mathbf{x}) g^2(\mathbf{x})) - \chi((-1)^kB f^2(\mathbf{x})g^3(\mathbf{x})) - \chi(AB f^3(\mathbf{x}) g^3(\mathbf{x})) \\
=& \frac{1}{4} \sum_{\mathbf{x} \in\mathbb{P}^{n-1}(\F_q)} \chi(f^2(\mathbf{x}) g^2(\mathbf{x})) + \chi((-1)^kA f(\mathbf{x}) g^2(\mathbf{x})) - \chi((-1)^k B f^2(\mathbf{x})g(\mathbf{x})) - \chi(AB f(\mathbf{x}) g(\mathbf{x})).
\end{align*}
Next, we carefully analyze the four terms in the character sum above.

\medskip 

\textbf{Main Term.} We first examine the character sum involving $f^2(\mathbf{x})g^{2}(\mathbf{x})$ which will dominate the remaining terms.
\begin{equation}\label{eq:1st}
\frac{1}{4} \sum_{\mathbf{x}\in\mathbb{P}^{n-1}(\F_q)} \chi(f^2(\mathbf{x})g^2(\mathbf{x})) 
= \frac{1}{4} \sum_{x\in\mathbb{P}^{n-1}(\F_q)} 1 -  \frac{1}{4} \sum_{\substack{\mathbf{x}\in\mathbb{P}^{n-1}(\F_q) \\ f(\mathbf{x}) g(\mathbf{x})=0}} 1 = \frac{q^{n-1}}{4} + O(q^{n-2}).  
\end{equation}

\textbf{Secondary terms.} By Lemma~\ref{lem1},
\begin{equation}\label{eq:4th}
\left|\sum_{\mathbf{x}\in\mathbb{P}^{n-1}(\F_q)}\chi(f(\mathbf{x}) g(\mathbf{x}))\right| \leq 3\cdot 8^{n+1} q^{(2n-3)/2} +2q^{n-2}=O(q^{n-\frac{3}{2}}).    
\end{equation}

By symmetry, character sums involving $f(\mathbf{x}) g^2(\mathbf{x})$ and $f^2(\mathbf{x}) g(\mathbf{x})$ behave similarly; we only analyze the first one. We first rewrite the character sum:
$$
\sum_{\mathbf{x}\in\mathbb{P}^{n-1}(\F_q)} \chi(f(\mathbf{x}) g^2(\mathbf{x})) = 
\sum_{\mathbf{x}\in\mathbb{P}^{n-1}(\F_q)} \chi(f(\mathbf{x})) - \sum_{\substack{\mathbf{x}\in \mathbb{P}^{n-1}(\F_q) \\ g(\mathbf{x})= 0 }} \chi(f(\mathbf{x})).
$$
Since $C=\{f=0\}$ is smooth, we can apply Lemma~\ref{lem:Katz} with $d=2$:
$$
\left|\sum_{\mathbf{x}\in\mathbb{P}^{n-1}(\F_q)} \chi(f(\mathbf{x}))\right| \leq \frac{q^{n/2}}{q-1}=O(q^{n-2}).
$$
Note that  Lemma~\ref{lem:Katz} is stated for character sums over affine spaces, but we can apply it to the projective space in our setting. Indeed, $\chi(f(\mathbf{x}))=\chi(f(\lambda\mathbf{x}))$ for any non-zero scalar $\lambda$ as $\deg(f)=2$ and $\chi$ is the quadratic character. We divide by $q-1$ to account for this scaling by $\lambda$.

We also have the trivial bound:
$$
\left|\sum_{\substack{\mathbf{x}\in \mathbb{P}^{n-1}(\F_q) \\ g(\mathbf{x})= 0 }} \chi(f(\mathbf{x})) \right| \leq O(q^{n-2}).
$$
Therefore,
\begin{equation}\label{eq:2nd}
\sum_{\mathbf{x}\in\mathbb{P}^{n-1}(\F_q)} \chi(f(\mathbf{x}) g^2(\mathbf{x})) = O(q^{n-2}).
\end{equation}
By symmetry, we also have
\begin{equation}\label{eq:3rd}
\sum_{\mathbf{x}\in\mathbb{P}^{n-1}(\F_q)} \chi(f^2(\mathbf{x}) g(\mathbf{x})) = O(q^{n-2}).
\end{equation}

Combining the estimates in \eqref{eq:1st}, \eqref{eq:4th}, \eqref{eq:2nd}, and \eqref{eq:3rd}, we conclude that
$$
\#S_{f, g}= \frac{q^{n-1}}{4} + O(q^{n-\frac{3}{2}}), 
$$
as desired. $\hfill\square$

\section*{Acknowledgments}
The authors thank Seoyoung Kim and Semin Yoo for helpful discussions. The authors also thank the anonymous referees for their valuable comments and suggestions. The research of the second author was supported in part by an NSERC fellowship.

\bibliographystyle{abbrv}
\bibliography{main}

\end{document}